\documentclass[12pt]{amsart}
\usepackage{amsmath, amsthm, amsfonts, amssymb}

\newtheorem{theorem}{Theorem}[section]

\newtheorem{lemma}[theorem]{Lemma}
\theoremstyle{definition}

\newtheorem{notation}[theorem]{Notation}

\newtheorem{remark}[theorem]{Remark}

\begin{document}
	
\title[SOLVABILITY BY  The SUM OF ELEMENT ORDERS]{A Criterion for Solvability of a Finite \\ Group by the Sum of Element Orders}	
\author[  M. Baniasad, B. Khosravi ]{ Morteza Baniasad Azad \& Behrooz Khosravi }
\address{ Dept. of Pure  Math.,  Faculty  of Math. and Computer Sci. \\
Amirkabir University of Technology (Tehran Polytechnic)\\ 424,
Hafez Ave., Tehran 15914, Iran \newline }
\email{ baniasad84@gmail.com}
\email{ khosravibbb@yahoo.com}

\thanks{}
\subjclass[2000]{20D60, 20F16}

\keywords{Sum of element orders, solvable group, element orders.}

\begin{abstract}
		Let $G$ be a finite group and  $\psi(G) = \sum_{g \in G} o(g)$, where $o(g)$ denotes
		the order of $g \in G$.
	In [M. Herzog, et. al.,  Two new criteria for solvability of finite groups, J. Algebra, 2018],
	the authors  put forward the following conjecture:\\
	\textbf{Conjecture.}
\textit{If $G$ is a group of order $n$ and	$\psi(G)>211\psi(C_n)/1617 $, where $C_n$ is the cyclic group of order $n$,	then $G$ is solvable.}
	
	In this paper we prove the validity of this conjecture.
\end{abstract}

\maketitle
\section{\bf Introduction}
In this paper all groups are finite. The cyclic group of order $n$ is denoted by $C_n$.  Let $ \psi(G) = \sum_{g \in G} o(g)$,  the sum of element orders in a  group $G$.
The function $\psi(G)$ was introduced  by Amiri, Jafarian
and  Isaacs \cite{amiri2009sums}. 
We can see that $\psi(C_{p^{\alpha}})=\frac{p^{2\alpha+1}+1}{p+1}$, where $\alpha \in \mathbb{N}$.
In \cite{herzog2017exact}, an exact upper bound for sums of element orders in non-cyclic
finite groups is given. 
In \cite{herzog2018two}, the authors give two new criteria for solvability of finite groups.
They proved that, 	if $G$ is a group of order $n$ and
$\psi(G)\geq \psi(C_n)/6.68 $,
then $G$ is solvable.

From the observation that $A_5$ satisfies $\psi(A_5)=211$ and $\psi(C_{60})=1617$,  in \cite{herzog2018two} 
they put forward the following conjecture:\\
\textbf{Conjecture.} \textit{If $G$ is a group of order $n$ and
	$\psi(G) > \frac{211}{1617} \psi(C_n)$,
	then $G$ is solvable.}
As the main result of this paper we prove the validity of this conjecture.

 For the proof of this result, we need the following
 lemmas.
\begin{lemma}\cite[Corollary B]{amiri2009sums} \label{sumsylow}
	Let $P \in {\rm Syl}_p(G),$ and assume that $P \unlhd G$ and that $P$ is cyclic. Then
	$\psi(G) \leq \psi(P)\psi(G/P)$, with equality if and only if $P$ is central in $G$.
\end{lemma}
\begin{lemma}
	\cite[Proposition 2.6]{herzog2018two} \label{prop}
	Let $H$ be a normal subgroup of the finite group $G$. Then
	$\psi(G) \leq \psi(G/H)|H|^2$.
\end{lemma}
\begin{lemma}\cite[Lemma 2.1]{Amiri2011zbMATH05906990} \label{sumdirect}
	If $G$ and $H$ are finite groups, then $\psi(G \times H) \leq \psi(G)\psi(H)$. Also,
	$\psi(G \times H) = \psi(G)\psi(H)$ if and only if $\gcd(|G|, |H|) = 1$.
\end{lemma}
\begin{lemma}\cite[Proposition 2.5]{herzog2017exact} \label{2p}
Let $G$ be a finite group and suppose that there exists $x\in G$ such that $|G:\langle x\rangle|<2p$,
where $p$ is the maximal prime divisor of $|G|$. Then one of the following holds:

(i) $G$ has a normal cyclic Sylow $p$-subgroup,

(ii) $G$ is solvable and $\langle x\rangle$
is a maximal subgroup of $G$ of index either $p$ or $p+1$.
\end{lemma}
\begin{lemma}\cite[Theorem 1]{herzog2018two} \label{solpower}
	Let $G$ be a finite group of order $n$ containing a subgroup $A$ of prime power
	index $p^s$. Suppose that $A$ contains a normal cyclic subgroup $B$ satisfying the following
	condition: $A/B$ is a cyclic group of order $2^r$ for some non-negative integer $r$. Then $G$ is
	a solvable group.
\end{lemma}
\begin{lemma} \cite[Theorem]{herstein} \label{herstein}
	Let $G$ be a finite group, $A$ an Abelian subgroup of $G$. If
	$A$ is a maximal subgroup of $G$ then $G$ is solvable.
\end{lemma}
\begin{lemma}\cite[Lemma 9.1]{isaa} \label{isaa}
	Let $G$ be a group, and suppose that $G / Z ( G )$ is simple. Then
	$G / Z(G)$ is non-Abelian, and $G^{\prime}$ is perfect. Also, $G^{\prime}/ Z (G^{\prime})$ is isomorphic to
	the simple group $G / Z(G)$.
\end{lemma}
\begin{lemma}\cite[Theorem 2.20]{isaa} (Lucchini) \label{lucchini}
	Let $A$ be a cyclic proper subgroup of a finite
	group $G$, and let $K={\rm core}_G(A)$. Then $| A : K | < | G : A |$, and in particular,
	if $|A| > |G : A|$ , then $K>1$.
\end{lemma}
\begin{lemma}\cite[Theorem 3.1]{hallsylow} \label{sylow}
	 If $n = 1 + rp$, with $1 < r < (p + 3)/2$ there is not a
	group $G$ with $n$ Sylow $p$-subgroups unless $n = q^t$ where $q$ is a prime, or $r = (p - 3)/2$ and
	$p > 3$ is a Fermat prime.
\end{lemma}
For a prime number $q$, by $n_q(G)$ or briefly $n_q$, we denote the number of Sylow $q$-subgroups of $G$. Also, 
the set of all Sylow $q$-subgroups of $G$ is denoted by ${\rm Syl}_q(G)$. 	If $n$ is an integer, then $\pi(n)$ is the set of all prime divisors of $n$.
If $G$ is a finite group, then  $\pi(|G|)$ is denoted by $\pi(G)$.
\begin{lemma}\cite[Theorem 3.2]{hallsylow} \label{21}
 There is no group $G$ with $n_3 = 22$, with $n_5 = 21$, or with
$n_p=1+3p$ for $p\geqslant 7$.
\end{lemma}

\begin{notation} \cite[Notation 2.1 and Notation 2.2]{herzog2018two}
	Let $ \{ q_1, q_2, q_3, \cdots \} $ be the set of \textbf{all} primes in an increasing order: 
	$2 = q_1 < q_2 < q_3 < \cdots $. Let also $q_0 = 1$.
	If $r, s$ are two  positive integers, we define the functions $f(r)$ and $h(s)$ as follows:
	\[f(0)=1, \quad f(r)=\prod_{i=1}^{r} \dfrac{q_i}{q_i+1}; \qquad h(1)=2, \quad h(s)=f(s-1)q_s.\]
\end{notation}


\section{\bf  Preliminary Results}
\begin{lemma}\label{aval}
	Let $G$ be a group of order $n={p_1}^{\alpha_1}\cdots{p_k}^{\alpha_k}$, where ${p_1}, \cdots,{p_k}$ are distinct primes. Let $\psi(G)>\dfrac{r}{s} \psi(C_n)$, for some integers $r, s$. Then  there exists a cyclic subgroup $ \langle{x}\rangle$ such that
	\[ [G:\langle{x}\rangle]< \dfrac{s}{r} \cdot \dfrac{p_1+1}{p_1}\cdots \dfrac{p_k+1}{p_k}. \]
\end{lemma}
\begin{proof}
		We have
		\begin{align*} \label{sumppp}
		\psi(G)&>\dfrac{r}{s}\psi(C_n)=\dfrac{r}{s}\psi(C_{{p_1}^{\alpha_1}\cdots{p_k}^{\alpha_k}})
		=\dfrac{r}{s}\psi(C_{{p_1}^{\alpha_1}})\cdots\psi(C_{{p_k}^{\alpha_k}}) \\
		&=\dfrac{r}{s} \cdot\dfrac{{p_1}^{2\alpha_1+1}+1}{p_1+1} \cdots \dfrac{{p_k}^{2\alpha_k+1}+1}{p_k+1}
		>\dfrac{r}{s} \cdot\dfrac{{p_1}^{2\alpha_1+1}}{p_1+1} \cdots \dfrac{{p_k}^{2\alpha_k+1}}{p_k+1}\\
		&=\dfrac{r}{s} \cdot\dfrac{p_1}{p_1+1} \cdots \dfrac{p_k}{p_k+1} n^2.
		\end{align*}
	So	it follows  that there exists $x\in G $ such that
		\[o(x)>\dfrac{r}{s} \cdot\dfrac{p_1}{p_1+1} \cdots \dfrac{p_k}{p_k+1} n.\]
		 We conclude that
		\[ [G:\langle{x}\rangle]< \dfrac{s}{r} \cdot \dfrac{p_1+1}{p_1}\cdots \dfrac{p_k+1}{p_k},\]
		and we get the result.
\end{proof}

\begin{lemma} \label{pa}
Let $p$ be a prime number and $a,b>0$. Then $\psi(C_{p^{a+b}})\geq \psi(C_{p^a})\psi(C_{p^b})$.
\end{lemma}
\begin{proof} We have
	\begin{align*}
\psi(C_{p^{a+b}}) & \geq \psi(C_{p^a})\psi(C_{p^b}) \Leftrightarrow \dfrac{p^{2a+2b+1}+1}{p+1} \geq \dfrac{p^{2a+1}+1}{p+1} \cdot \dfrac{p^{2b+1}+1}{p+1} \\
&\Leftrightarrow p^{2a+2b+2}+p+p^{2a+2b+1}+1 \geq p^{2a+2b+2} + p^{2b+1} + p^{2a+1}+1 \\
&\Leftrightarrow p^{2a+2b+1}+1 \geq  p^{2b} + p^{2a} \\
&\Leftrightarrow (p^{2a}-1)(p^{2b}-1) \geq  0, 
	\end{align*}
and we 	get the result.
\end{proof}
\begin{lemma} \label{2m}
(a)	Let $p \in \{2, 3, 5\}$  and $a>0$. Then $ p^{2a} >  \frac{13}{12} \psi(C_{p^a}) $.\\
(b) Let $\pi(m) \subseteq\{2, 3, 5\}$ and $m\geq 2$. Then $ m^2 >  \frac{13}{12}  \psi(C_m) $.
\end{lemma}
\begin{proof} 
(a)	We proceed by induction on $a$. If $a=1$, then
	\[2^2 > \dfrac{13}{12}  \psi(C_2)=3.25, \quad 3^2 > \dfrac{13}{12}  \psi(C_3)\approx7.59,  \quad 5^2 > \dfrac{13}{12}  \psi(C_5)=22.75.\]
	Therefore by the inductive hypothesis we have
	\begin{align*}
	p^{2(a+1)}=p^{2a}p^2&> \dfrac{13}{12} \psi(C_{p^a})p^2=\dfrac{13}{12} \dfrac{p^{2a+1}+1}{p+1}p^2=\dfrac{13}{12}\dfrac{p^{2a+3}+p^2}{p+1}\\
	&>\dfrac{13}{12} \dfrac{p^{2a+3}+1}{p+1}= \dfrac{13}{12} \psi(C_{p^{a+1}}),
	\end{align*}
	as wanted.\\
(b)	If $m=2^r3^s5^t\geq2$, then by  (a) we have
\begin{align*}
m^2=2^{2r}3^{2s}5^{2t}&>\dfrac{13}{12} \psi(C_{2^r}) \dfrac{13}{12} \psi(C_{3^s}) \dfrac{13}{12} \psi(C_{5^t}) \\
 &> \dfrac{13}{12} \psi(C_{2^r})  \psi(C_{3^s})  \psi(C_{5^t})=\dfrac{13}{12} \psi(C_{m}),
\end{align*}
as wanted.
\end{proof}
We know that $h(6)=\frac{5005}{1152}$. Now we state a lemma similar to Lemmas 2.4 and 2.5 in \cite{herzog2018two}.
\begin{lemma} \label{herzog}
Let  $n = {p_1}^{\alpha_1} {p_2}^{\alpha_2} \cdots {p_r}^{\alpha_r}$  be a positive integer,
 where $p_{i}$ are primes, $p_1 < p_2 < \cdots < p_r = p$ and $\alpha_i>0$, for each $1 \leq i \leq r$.
 If $p \geq 13$, then
	$$\psi(C_n) \geq \dfrac{5005}{1152} \dfrac{n^2}{p+1}.$$
\end{lemma}
\begin{proof}
	The proof is similar to the proof of Lemmas 2.4 and 2.5 in \cite{herzog2018two}.
\end{proof}

\section{\bf  Proof of the Conjecture}

\begin{proof}[Proof of the conjecture]
	We prove by induction on $|\pi(G)|$ that 
	if $G$ is a group of order $n = {p_1}^{\alpha_1} {p_2}^{\alpha_2} \cdots {p_r}^{\alpha_r}$,
	where $p_{i}$ are primes, $p_1 < p_2 < \cdots < p_r = p$ such that $\alpha_i>0$, for each $1 \leq i \leq r$, and also
	 $\frac{211}{1617}\psi(C_n)<\psi(G)$, then $G$ is a solvable group.
	 
	 If $|\pi(G)|=1$, then $G$ is a $p$-group, therefore $G$ is solvable.
	 If $|\pi(G)|=2$, then by Burnside's $p^aq^b$-theorem  $G$ is solvable.
Assume that $|\pi(G)|\geq3$ and the theorem holds for each group $H$ such that $|\pi(H)|<|\pi(G)|$. Now we consider the following two cases:

\textbf{Case(I).}
 If $G$ has a  normal cyclic Sylow subgroup $Q$, then
by Lemma \ref{sumsylow}  we have
$\psi(G) \leq \psi(Q)\psi(G/Q)$. Using Lemma \ref{sumdirect} and the assumptions we have
\[\dfrac{211}{1617} \psi(C_{|G/Q|})\psi(C_{|Q|})=\dfrac{211}{1617} \psi(C_n)<\psi(G)\leq\psi(Q)\psi(G/Q)=\psi(C_{|Q|})\psi(G/Q).\]
Therefore $\frac{211}{1617}\psi(C_{|G/Q|})<\psi(G/Q)$ and $|\pi(G/Q)|<|\pi(G)|$. By the inductive hypothesis, $G/Q$ is solvable
and so $G$ is a solvable group. 

\textbf{Case(II).} Let $G$ have no normal cyclic Sylow  subgroup.

We note that if there exists $x\in G$ such that $|G:\langle x\rangle|<2p$, then $G$ is solvable, by Lemma \ref{2p}.

	If  $p\geq13$, then  by Lemma \ref{herzog}, we have 
	\[\psi(G) > \dfrac{211}{1617}\psi(C_n) >\dfrac{211}{1617} \cdot \dfrac{5005}{1152} \dfrac{n^2}{p+1}. \]
	Thus	there exists $x \in G$ such that $o(x)>\frac{211}{1617} \cdot\frac{5005}{1152} \frac{n}{p+1}$. Therefore
	\[[G:\langle x \rangle] < \dfrac{1617}{211} \cdot \dfrac{1152}{5005} (p+1) \leq \dfrac{1617}{211} \cdot \dfrac{1152}{5005} \cdot \dfrac{14}{13} p<2p, \]
	and by the above discussion $G$ is a solvable group.

Therefore $\pi(G) \subseteq \{2, 3, 5, 7, 11\}$, where $ 3 \leq|\pi(G)| \leq 5$. 
By Feit-Thompson, theorem every finite group of odd order is solvable, and so in the sequel,  we  assume that $2\in \pi(G)$. 

If $3 \notin \pi(G)$, then let $R$ be the solvable radical of $G$. If $R \ne G$, then $G/R$ is non-solvable and so there exists a non-Abelian simple  group $S$ such that $|S|$ is a divisor of $|G|$.
Since the Suzuki groups are the only non-Abelian simple groups whose orders are prime to $3$ and we know that there exists no Suzuki simple group $S$, where $\pi(S) \subseteq \{2, 5, 7, 11\}$,
we get a contradiction and so $G$ is a solvable group. So in the sequel we assume that $3 \in \pi(G)$.
Thus $\{2,3\} \subseteq \pi(G) \subseteq  \{2, 3, 5, 7, 11\}  $. 
Now we consider the following cases:
\\
\textbf{Case 1.} Let $\pi(G)=\{2, 3, 5\}$. Then $|G|=2^{\alpha_1}3^{\alpha_2}5^{\alpha_3}$. In this case we have
	\begin{align} \label{sum235}
	\psi(G)>\dfrac{211}{1617}\psi(C_{|G|})>\dfrac{211}{1617}\cdot \dfrac{2^{2{\alpha_1}+1}}{2+1} \cdot
	\dfrac{3^{2{\alpha_2}+1}}{3+1} \cdot \dfrac{5^{2{\alpha_3}+1}}{5+1}=\dfrac{211}{1617} \cdot \dfrac{5}{12}n^2.
	\end{align}
	It follows  that there exists $x\in G $ such that $o(x)>\frac{211}{1617} \cdot \frac{5}{12}n$. We conclude that
	$|G:\langle x \rangle|<\frac{1617}{211} \cdot \frac{12}{5}<19$. By Lemma \ref{2p}, we have $[G:\langle x \rangle]=10, 12, 15, 16$ or $18$.
	Now we consider each possibility for $[G:\langle x \rangle]$.
	\begin{itemize}
		\item Let $[G:\langle x \rangle]=10$.
		If $P_3 \in \rm{Syl}_3(G)$, then $ P_3 \leq  \langle x \rangle$ and so $P_3$  is a cyclic Sylow
		$3$-subgroup of $G$. Then
		 $\langle x \rangle \leq N_G(P_3)$ and \[10 = [G : N_G(P_3)][N_G(P_3) : \langle x \rangle] = (1 + 3k)[N_G(P_3) : \langle x \rangle].\]
		Since $P_3$ is not a normal subgroup of $G$,
		 we have $k = 3$ and $N_G(P_3) = \langle x \rangle$.
		 We claim that $\langle x \rangle$ is a maximal subgroup of $G$.
		  If $L$ is a subgroup of $G$ such
		that $N_G(P_3) < L$, then $N_L(P_3)=N_G(P_3) \cap L= N_G(P_3)$. Since $[L : N_L(P_3)] = 1 + 3m>1$ and
		 $10 = [G :  N_L(P_3)] =[G : L][L :  N_L(P_3)]$, we have $[G : L] = 1$ and so $G = L$. Hence $ \langle x \rangle$ is an Abelian maximal subgroup
		of $G$. By Lemma \ref{herstein}, $G$ is a solvable group.
		\item Let $[G:\langle x \rangle]=12$.
			If $P_5 \in \rm{Syl}_5(G)$, then $P_5 \leq  \langle x \rangle$ and so  is a cyclic Sylow
			$5$-subgroup of $G$. Then
			 $\langle x \rangle \leq N_G(P_5)$ and $$12 = [G : N_G(P_5)][N_G(P_5) : \langle x \rangle] = (1 + 5k)[N_G(P_5) : \langle x \rangle].$$
			Since $P_5$ is not a normal subgroup of $G$, $[G : N_G(P_5)]=6$ and $[N_G(P_5) : \langle x \rangle]=2$.
	Thus we have the following  series:
	\[	H={\rm core}_G(\langle x \rangle) \leq \langle x \rangle \leq N_G(P_5) \leq G.	\]	
	If $G/H$ is a solvable group, then we get the result. So let $G/H$ be a non-solvable group.
	By Lemma \ref{lucchini}, $[\langle x \rangle : H] < [G : \langle x \rangle] = 12$. 
	Also, $P_5$ is not a normal subgroup of $G$ and so $P_5 \nleq H$, i.e.
	 $ 5 \mid [\langle x \rangle : H]$.  Hence  either
	$ [\langle{x}\rangle: H] = 5$ or
	$[\langle{x}\rangle : H] = 10$. 
\\	
$\blacktriangleright$	
	If $[\langle{x}\rangle : H] = 5$, then $|G/H| = 60$ and so $G/H\cong A_5$. Therefore $H$ is a maximal normal subgroup of $G$ and $H\leq C_G(H)$. On the other hand,
	$C_G(H)\trianglelefteq N_G(H)=G$. By the definition of maximal normal subgroup, we have either $C_G(H)=H$ or $C_G(H)=G$.

	If $C_G(H)=H$, then by NC-theorem
	\[N_G(H)/C_G(H)=G/H\cong A_5 \hookrightarrow {\rm Aut(H)},\] which is a contradiction, since $H$ is cyclic.
	
	If $C_G(H)=G$, then $H\leq Z(G)$. Since $H$ is a maximal normal subgroup of $G$ and $G/H$ is non-solvable it follows that $H=Z(G)$
and so  $G/Z(G)\cong A_5$.  Therefore $G=G^{\prime}Z(G)$ and
	 \begin{align}
\dfrac{G^{\prime}Z(G)}{Z(G)}  \cong A_5  \Longrightarrow \dfrac{G^{\prime}}{G^{\prime} \cap Z(G)} \cong A_5 \label{c3}.
	 \end{align}
	
On the other hand,	 using Lemma \ref{isaa}, $G^{\prime}$ is perfect and $G^{\prime}/Z(G^{\prime})\cong G/Z(G)\cong A_5$.	
	Therefore $G^{\prime}$ is a central
	extension of $Z(G^{\prime})$ by $A_5$. Since the Shur multiplier of $A_5$ is $2$ and $G^{\prime}$ is perfect, we get that
	 either $G^{\prime}\cong A_5$ or $G^{\prime}\cong {\rm SL}(2,5)$.
Now, we consider two cases:
	\begin{itemize}
		\item Let $G^{\prime} \cong A_5$. Using (\ref{c3}), we have $G^{\prime} \cap Z(G)=1$, hence
		$G\cong G^{\prime} \times Z(G) \cong A_5 \times Z(G) $ and $|Z(G)|=2^{{\alpha_1}-2}3^{{\alpha_2}-1}5^{{\alpha_3}-1}$.
		By Lemmas \ref{sumdirect} and \ref{pa}, we have
		\begin{align*}
	\qquad \qquad	\psi(G) &= \psi(A_5 \times Z(G)) \leq \psi(A_5) \psi(Z(G)) \\
	&=211 \psi(C_{2^{{\alpha_1}-2}}) \psi(C_{3^{{\alpha_2}-1}}) \psi(C_{5^{{\alpha_3}-1}}) \\
		&\leq 211 \dfrac{\psi(C_{2^{{\alpha_1}}})}{\psi(C_{4})} \dfrac{C_{\psi(3^{{\alpha_2}}})}{\psi(C_{3})} \dfrac{\psi(C_{5^{{\alpha_3}}})}{\psi(C_{5})}= \dfrac{211}{1617} \psi(C_{n}),
		\end{align*}
		
		which is a contradiction.
		
		\item Let $G^{\prime}\cong {\rm SL}(2,5)$. Then, using (\ref{c3}), we have $|G^{\prime} \cap Z(G)|=2$. Therefore
		\begin{align*}
\qquad	\qquad	\dfrac{G}{G^{\prime}\cap Z(G)} \cong \dfrac{G^{\prime}Z(G)}{G^{\prime}\cap Z(G)} \cong
		 \dfrac{G^{\prime}}{G^{\prime}\cap Z(G)} \times \dfrac{Z(G)}{G^{\prime}\cap Z(G)} \cong A_5 \times C_m,
		\end{align*}
		where $C_m$ is a cyclic group of order $m=n/120$. Thus Lemmas \ref{prop} and \ref{sumdirect}, imply that
		\begin{align*}
	\qquad \qquad	\psi(G) &\leq \psi(\dfrac{G}{G^{\prime}\cap Z(G)})|G^{\prime}\cap Z(G)|^2 =4 \psi(A_5 \times C_m)\\
		        & \leq 4 \psi(A_5)\psi(C_m) =4\cdot 211 \psi(C_{m}),
		\end{align*}
		Using (\ref{sum235}), $	\psi(G)>\frac{211}{1617}\cdot  \frac{5}{12} n^2=\frac{211}{1617}\cdot  \frac{5}{12} 120^2m^2$. Therefore
		\begin{align*}
\qquad \qquad
		\dfrac{211}{1617}\cdot  \dfrac{5}{12} 120^2m^2 & < 4\cdot 211 \psi(C_{m}) \Rightarrow 5\cdot120^2m^2<4\cdot 1617 \cdot 12 \psi(C_{m}) \\ 
		&\Rightarrow 72000 m^2< 77616 \psi(C_{m}) < 78000 \psi(C_{m}).
		\end{align*}
		Therefore  $12 m^2< 13 \psi(C_{m})$. Using Lemma \ref{2m}, we have $m=1$. Hence $|G|=120$. On the other hand,
				$G^{\prime}\cong {\rm SL}(2,5)$, it follows that $G\cong {\rm SL}(2,5)$, thus $\psi(G)=\psi({\rm SL}(2,5))=663$. Hence by our assumptions we have
	$663> \frac{211}{1617}\psi(120)>824$, which is a contradiction.	
	\end{itemize}
	$\blacktriangleright$	If $[\langle{x}\rangle : H] = 10$, then $G/H$ is a non-solvable group of order $120$. Therefore using the list of
	such groups ($\rm{SL}(2,5), S_5, C_2 \times A_5$) and their
	$\psi$-values ($663, 471, 603$), we have $\psi(G/H) \leq 663$. By Lemma \ref{prop} we have
	$\psi(G) \leq \psi(G/H)|H|^2 \leq 663 (n/120)^2$.
	Using (\ref{sum235}), $	\psi(G)>\frac{211}{1617}\cdot  \frac{5}{12} n^2$. Therefore
	\[\dfrac{211}{1617}\cdot  \dfrac{5}{12} n^2< 663 \dfrac{n^2}{120^2},\]
	which is a contradiction.	
 \item Let $[G:\langle x \rangle]=15$. Then a Sylow
 $2$-subgroup of $G$ is cyclic and hence $G$
 has a normal $2$-complement, therefore by the Feit-Thompson Theorem $G$
  is solvable.

\item Let $[G:\langle x \rangle]=16$. Then by Lemma \ref{solpower}, $G$ is a solvable group.

 \item Let $[G:\langle x \rangle]=18$. 
 	If $P_5 \in \rm{Syl}_5(G)$, then $P_5 \leq  \langle x \rangle$. Therefore  $\langle x \rangle \leq N_G(P_5)$ and 
 	 $18 = [G : N_G(P_5)][N_G(P_5) : \langle x \rangle] = (1 + 5k)
 [N_G(P_5) : \langle x \rangle]$.
 Since $P_5 \ntrianglelefteq G$, we have $k = 1$, thus $[G : N_G(P_5)]=6$ and $[N_G(P_5) : \langle x \rangle]=3$.
 Let $ H={\rm core}_G(\langle x \rangle)$. By Lemma \ref{lucchini}, $[\langle x \rangle : H] < [G : \langle x \rangle] = 18$.
 Also,
 	$ 5 \mid [\langle x \rangle : H]$, because $P_5 \ntrianglelefteq G$. 
 	If $G/H$ is a solvable group, then we get the result. Let $G/H$ be non-solvable. Therefore $ [\langle{x}\rangle: H] =  10$ and $|G/H|=180$. Thus $G/H\cong {\rm GL}(2,4)$ and we have $\psi(G/H)=\psi({\rm GL}(2,4))=1237$. Thus Lemma \ref{prop}  implies that
 	$ 		\psi(G) \leq \psi(G/H)|H|^2 =1237(n/180)^2. 	$
 	Using    (\ref{sum235}), 
 	we have the following contradiction:
 	\[ 	\dfrac{211}{1617} \cdot \dfrac{5}{12} n^2 < 1237(\dfrac{n}{180})^2.\]
	\end{itemize}
So if $\pi(G)=\{2, 3, 5\}$, then $G$ is solvable.
\\
\textbf{Case 2.} Let $\pi(G)=\{2, 3, 7\}$. Then
	\begin{align} \label{sum237}
	\psi(G)>
	\dfrac{211}{1617}\dfrac{2^{2{\alpha_1}+1}}{2+1}
	\dfrac{3^{2{\alpha_2}+1}}{3+1}\dfrac{7^{2{\alpha_3}+1}}{7+1}=\dfrac{211}{1617} \cdot \dfrac{7}{16} n^2.
	\end{align}
	Therefore there exists $x\in G $ such that $o(x)>\frac{211}{1617} \cdot \frac{7}{16} n$. We obtain that
	$[G:\langle x \rangle]<\frac{1617}{211} \cdot \frac{16}{7} <18$. Using Lemma \ref{2p}, we have $|G:\langle x \rangle|=14$ or $16$.
	\begin{itemize}
		\item Let $[G:\langle x \rangle]=16$. Using Lemma \ref{solpower}, we get that $G$ is  solvable.
		\item Let $[G:\langle x \rangle]=14$.
		If $P_3$  is a  Sylow
		$3$-subgroup of $G$, then $\langle x \rangle \leq N_G(P_3)$ and $14 = [G : N_G(P_3)][N_G(P_3) : \langle x \rangle] = (1 + 3k)
		[N_G(P_3) : \langle x \rangle]$.
		We have $k = 2$, thus $[G : N_G(P_3)]=7$ and $[N_G(P_3) : \langle x \rangle]=2$.
		 Let $ H={\rm core}_G(\langle x \rangle)$.
		 By Lemma \ref{lucchini}, $[\langle x \rangle : H] < [G : \langle x \rangle] = 14$. 
		 Since $P_3 \ntrianglelefteq G$,  $ 3 \mid [\langle x \rangle : H]$.  Hence 
		 $ [\langle{x}\rangle: H] = 3, 6, 9$ or $12$.
		 If $G/H$ is a solvable group, then we get the result. Let $G/H$ be non-solvable. Therefore $ [\langle{x}\rangle: H] =  12$ and $|G/H|=168$. Thus $G/H\cong {\rm PSL}(2,7)$ and we have $\psi(G/H)=\psi({\rm PSL}(2,7))=715$. Thus Lemma \ref{prop}  implies that
		 $
		 \psi(G) \leq \psi(G/H)|H|^2 =715(n/168)^2.
		 $
		  	Using    (\ref{sum237}), 
		  	we have the following contradiction:
		 \begin{align*}
		 \dfrac{211}{1617} \cdot \dfrac{7}{16}n^2 < 715(\dfrac{n}{168})^2.
		 \end{align*}
\end{itemize}
So if $\pi(G)=\{2, 3, 7\}$, then $G$ is solvable.
\\
\textbf{Case 3.} Let $\pi(G)=\{2, 3, 5, 7\}$. Then by Lemma \ref{aval},
	there exists $x\in G $ such that 
	$[G:\langle x \rangle]<\frac{1617}{211} \cdot \frac{96}{35}<22$. Using Lemma \ref{2p}, we have
	 $[G:\langle x \rangle]=14, 15, 16, 18, 20$ or $21$. Let $P_5 \in {\rm  Syl}_5(G)$ and $P_7 \in {\rm  Syl}_7(G)$.
	\begin{itemize}
		\item Let $[G:\langle x \rangle]=14$. Then $P_5 \leq \langle x \rangle$ and $14 = [G : N_G(P_5)][N_G(P_5) : \langle x \rangle]$, 
		which is impossible since $P_5 \ntrianglelefteq G$.
		\item Let $[G:\langle x \rangle]=15$. Then $P_7 \leq \langle x \rangle$ and $P_7 \ntrianglelefteq  G$.
		Therefore 
		$$15 = [G : N_G(P_7)][N_G(P_7) : \langle x \rangle]= (1 + 7k)[N_G(P_7) : \langle x \rangle].$$
		Thus  $k=2$, which is impossible by Lemma \ref{sylow}.
		\item Let $[G:\langle x \rangle]=16$. Then by Lemma \ref{solpower}, we have $G$ is  solvable.
		\item Let $[G:\langle x \rangle]=18$ or $20$.  It is impossible, because $P_7 \leq \langle x \rangle$ and $P_7 \ntrianglelefteq  G$.
		\item Let $[G:\langle x \rangle]=21$. Then $P_5 \leq \langle x \rangle$ and 
		$$21 = [G : N_G(P_5)][N_G(P_5) : \langle x \rangle] =(1 + 5k)	[N_G(P_5) : \langle x \rangle].$$
		 We obtain that $n_5=21$ and by Lemma \ref{21}, we get a contradiction.
	\end{itemize}
Therefore $G$ is solvable, when $\pi(G)=\{2, 3, 5, 7\}$.\\
\textbf{Case 4.} Let $\pi(G)=\{2, 3, 11\}$, $\pi(G)=\{2, 3, 5, 11\}$ or $\pi(G)=\{2, 3, 7, 11\}$. Then by Lemma \ref{aval}, there exists $x \in G$ such that
 $[G:\langle x \rangle]<21$. Using Lemma \ref{2p}, $G$ is a solvable group. \\
\textbf{Case 5.} Let $\pi(G)=\{2, 3, 5, 7, 11\}$.
 Then by Lemma \ref{aval}, there exists $x \in G$ such that
$$[G:\langle x \rangle]< \dfrac{1617}{211} \cdot \dfrac{3}{2} \cdot \dfrac{4}{3} \cdot \dfrac{6}{5} \cdot \dfrac{8}{7} \cdot \dfrac{12}{11}<23.$$
Using Lemma \ref{2p}, we have  $[G:\langle x \rangle]=22$. Let $P_7 \in \rm{Syl}_7(G)$. Therefore 
\[ [G:\langle x \rangle]= [G:N_G(P_7)][N_G(P_7):\langle x \rangle]=(1+7k)[N_G(P_7):\langle x \rangle]=22. \]
Hence $k=3$  and by Lemma \ref{sylow}, we get a contradiction.

 The proof is now complete.
\end{proof}
About the equality in \cite[Conjecture 6]{herzog2018two},  we give the following remark:
\begin{remark}
	We note that if
	$G=A_5 \times C_m$, where $\gcd(30,m)=1$, then by Lemma \ref{sumdirect}, we have
	$\psi(G)=\psi(A_5 \times C_m)=\psi(A_5)\psi(m).$
	On the other hand,
	\[\dfrac{211}{1617}\psi(C_{60 m})=\dfrac{\psi(A_5)}{\psi(C_{60})}\psi(C_{60})\cdot \psi(C_{m})=\psi(A_5)\psi(C_{m}).\]
	Therefore $\psi(G)=\dfrac{211}{1617}\psi(C_{|G|})$.
\end{remark}


\begin{thebibliography}{10}

\bibitem{Amiri2011zbMATH05906990}
H.~{Amiri}, S.M.~{Jafarian Amiri},
\newblock {Sum of element orders on finite groups of the same order},
\newblock {\em {J. Algebra Appl.}}, 10(2) (2011)  187--190.

\bibitem{amiri2009sums}
H.~{Amiri}, S.M. {Jafarian Amiri}, I.M. {Isaacs},
\newblock {Sums of element orders in finite groups,}
\newblock {\em {Comm. Algebra}}, 37(9) (2009) 2978--2980.



\bibitem{hallsylow}
M.~{Hall} Jr.,
\newblock {On the number of Sylow subgroups in a finite group,}
\newblock {\em {J. Algebra}}, 7 (1967)  363--371.


\bibitem{herstein}
I.N.~{Herstein},
\newblock{A remark on finite groups,}
\newblock{\em {Proc. Amer. Math. Soc.}}, 9 (1958),  255--257.




\bibitem{herzog2017exact}
M.~{Herzog}, P.~{Longobardi}, M.~{Maj},
\newblock {An exact upper bound for sums of element orders in non-cyclic finite groups,}
\newblock {\em {J. Pure Appl. Algebra}}, 222 (7) (2018) 1628--1642.

\bibitem{herzog2018two}
M.~{Herzog}, P.~{Longobardi}, M.~{Maj},
\newblock{Two new criteria for solvability of finite groups,}
\newblock{\em {J Algebra}}, 511 (2018) 215--226.



\bibitem{isaa}
I.M.~{Isaacs}, \newblock{Finite Group Theory}, \newblock{\em {Amer. Math. Soc.}}, Providence, Rhode Island, 2008.
\end{thebibliography}
\end{document}